\documentclass[a4paper,12pt]{amsart}
\usepackage{amssymb}
\usepackage{ifthen}
\usepackage{graphicx}
\usepackage{float}
\usepackage{caption}
\usepackage{subcaption}
\usepackage{cite}
\usepackage{amsfonts}
\usepackage{amscd}
\usepackage{amsxtra}
\usepackage{color}

\newtheorem{thm}{Theorem}
\newtheorem{cor}{Corollary}
\newtheorem{lem}{Lemma}
\newtheorem{prop}{Proposition}

\newtheorem{conj}{Conjecture}
\newtheorem{prob}{Problem}
\newtheorem{remark}{Remark}

\theoremstyle{definition}
\newtheorem{defn}{Definition}[section]
\newtheorem{example}{Example}

\newenvironment{pf}[1][]{%
 \vskip 1mm
 \noindent
 \ifthenelse{\equal{#1}{}}%
  {{\slshape Proof. }}%
  {{\slshape #1.} }%
 }%
{\qed\bigskip}

\newcounter{alphabet}
\newcounter{tmp}
\newenvironment{Thm}[1][]{\refstepcounter{alphabet}%
\bigskip%
\noindent%
{\bf Theorem \Alph{alphabet}}%
\ifthenelse{\equal{#1}{}}{}{ (#1)}%
{\bf .} \itshape}{\vskip 8pt}

\makeatletter
\newcommand{\Ref}[1]{\@ifundefined{r@#1}{}{\setcounter{tmp}{\ref{#1}}\Alph{tmp}}}
\makeatother

\newenvironment{Lem}[1][]{\refstepcounter{alphabet}%
\bigskip%
\noindent%
{\bf Lemma \Alph{alphabet}}%
{\bf .} \itshape}{\vskip 8pt}

\newcommand{\IC}{{\mathbb C}}
\newcommand{\ID}{{\mathbb D}}

\def\be{\begin{equation}}
\def\ee{\end{equation}}

\newcommand{\bee}{\begin{enumerate}}
\newcommand{\eee}{\end{enumerate}}

\newcommand{\blem}{\begin{lem}}
\newcommand{\elem}{\end{lem}}
\newcommand{\bthm}{\begin{thm}}
\newcommand{\ethm}{\end{thm}}
\newcommand{\bcor}{\begin{cor}}
\newcommand{\ecor}{\end{cor}}
\newcommand{\beg}{\begin{example}}
\newcommand{\eeg}{\end{example}}
\newcommand{\begs}{\begin{examples}}
\newcommand{\eegs}{\end{examples}}
\newcommand{\bdefe}{\begin{defn}}
\newcommand{\edefe}{\end{defn}}
\newcommand{\bprob}{\begin{prob}}
\newcommand{\eprob}{\end{prob}}
\newcommand{\bques}{\begin{ques}}
\newcommand{\eques}{\end{ques}}
\newcommand{\bei}{\begin{itemize}}
\newcommand{\eei}{\end{itemize}}
\newcommand{\bcon}{\begin{conj}}
\newcommand{\econ}{\end{conj}}
\newcommand{\bcons}{\begin{conjs}}
\newcommand{\econs}{\end{conjs}}
\newcommand{\bprop}{\begin{prop}}
\newcommand{\eprop}{\end{prop}}
\newcommand{\br}{\begin{remark}}
\newcommand{\er}{\end{remark}}
\newcommand{\brs}{\begin{rems}}
\newcommand{\ers}{\end{rems}}
\newcommand{\bo}{\begin{obser}}
\newcommand{\eo}{\end{obser}}
\newcommand{\bos}{\begin{obsers}}
\newcommand{\eos}{\end{obsers}}
\newcommand{\bpf}{\begin{pf}}
\newcommand{\epf}{\end{pf}}
\newcommand{\ba}{\begin{array}}
\newcommand{\ea}{\end{array}}
\newcommand{\beq}{\begin{eqnarray}}
\newcommand{\beqq}{\begin{eqnarray*}}
\newcommand{\eeq}{\end{eqnarray}}
\newcommand{\eeqq}{\end{eqnarray*}}

\newcounter{minutes}
\divide\time by 60
\newcounter{hours}
\multiply\time by 60 \addtocounter{minutes}{-\time}

\begin{document}
\bibliographystyle{amsplain}
\title[Bohr's inequality for analytic functions $\sum_k b_k z^{kp+m}$ and harmonic functions]{Bohr's inequality for analytic functions $\sum_k b_k z^{kp+m}$ and harmonic functions}

\thanks{
File:~\jobname .tex,
          printed: \number\day-\number\month-\number\year,
          \thehours.\ifnum\theminutes<10{0}\fi\theminutes}


\author{Ilgiz R Kayumov, and Saminathan Ponnusamy }

\address{I. R Kayumov, Kazan Federal University, Kremlevskaya 18, 420 008 Kazan, Russia
}
\email{ikayumov@kpfu.ru }

\address{S. Ponnusamy,
Indian Statistical Institute (ISI), Chennai Centre, SETS (Society
for Electronic Transactions and Security), MGR Knowledge City, CIT
Campus, Taramani, Chennai 600 113, India.
}
\email{samy@isichennai.res.in, samy@iitm.ac.in}


\begin{abstract}
We determine the Bohr radius for the class of all functions $f$ of the form $f(z)=\sum_{k=1}^\infty a_{kp+m} z^{kp+m}$ analytic in the unit disk $|z|<1$
and satisfy the condition $|f(z)|\le 1$ for all $|z|<1$. In particular, our result also contains a solution to a recent conjecture of
Ali, Barnard and Solynin \cite{AliBarSoly} for the Bohr radius for odd analytic functions, solved by the authors in \cite{KayPon1}.
We consider a more flexible approach by introducing the $p$-Bohr radius for harmonic functions which in turn contains the classical Bohr radius as special case.
Also, we prove several other new results and discuss $p$-Bohr radius for the class of odd harmonic bounded functions.
\end{abstract}

\thanks{
}

\maketitle
\pagestyle{myheadings}
\markboth{I. R. Kayumov and S. Ponnusamy}{Bohr's inequality for symmetric functions and harmonic functions}

\section{Preliminaries}
In 1914,  H. Bohr \cite{Bohr-14} proved that if the power series $f(z)=\sum_{k=0}^{\infty} a_kz^k$  converges in the open unit disk
$\ID :=\{z\in\IC:\, |z|<1\}$ and $|f(z)|<1$ for all $z\in \ID$, then the majorant series
$M_f(r):=\sum_{k=0}^{\infty}|a_k|r^k $ is less than or equal to $1$ for all
$|z|=r\leq 1/6$. The largest $r\leq 1$ such that the above inequality holds is referred to as the Bohr radius for the unit disk case.
The fact that the constant $1/3$ is best possible was established independently by F. Wiener, M. Riesz and I. Schur.
  Other proofs of this result were later
obtained by Sidon and Tomic.   Bohr's idea naturally extends to  functions of several complex variables and thus, a variety of results
on Bohr's theorem in higher dimension appeared recently. In this contexts and in other respects, we suggest the reader to refer
\cite{Abu,Abu2, Abu4, AAD2,AizenTark-01-2,Aizen-00-1,Aizen-05-3,BaluCQ-2006,BenDahKha,BoasKhavin-97-4,PaulSingh-04-11,Pop} and
the references there. For a detailed account of the development, we refer to the recent survey article
on this topic \cite{AAPon1} and the references therein. More recently, the present authors obtained the following result
as a corollary to a general result for symmetric functions and thereby settling the recent conjecture of
Ali et al. \cite{AliBarSoly}.

\begin{Thm}{\rm (\cite[Corollary 1]{KayPon1})}\label{KP2-odd1}
If $f(z)=\sum_{k=0}^{\infty} a_{2k+1}z^{2k+1}$ is odd analytic function in $\ID$ and $|f(z)| \leq 1$ in $\ID$, then
$M_f(r) \leq 1$  for $r \leq r_2=0.789991\ldots .$
The extremal function has the form $z(z^2-a)/(1-az^2)$.
\end{Thm}

Motivated by the work of Ali et al.\cite{AliBarSoly} and Theorem \Ref{KP2-odd1}, we raise the following

\bprob {\rm{(\cite{AliBarSoly})}\label{KP2-prob1}}
Given $p \in \mathbb{N}$ and $0\leq m \leq p$, determine the Bohr radius for the class of functions
$f(z)=\sum_{k=0}^{\infty} a_{pk+m}z^{pk+m}$ analytic in $\ID$
and $|f(z)| \leq 1$ in $\ID$.
\eprob

The case $m=1$ has been handled by the authors in \cite{KayPon1}.

One of the aims of this article is to solve this problem  completely  and present several of its consequences.
As remarked in \cite{KayPon1}, it brings serious difficulties because if we use sharp the inequalities
$|a_n| \leq 1-|a_0|^2$ $(n\geq 1)$ simultaneously (as in the classical case) then we will not be able to obtain
sharp result due to the fact that in the extremal case  $|a_0|<1$.
Also it is important that in the classical case there is no extremal function while in our case there is.



The paper is organized as follows. The solution to Problem \ref{KP2-prob1} is presented in Section \ref{KP2-sec2a} (see Theorem \ref{KP2-th3}) and its proof
is given in Section \ref{KP2-sec2}. In Section \ref{KP3-sec4}, we introduce the concept of $p$-Bohr radius for the class of bounded harmonic functions defined on the
unit disk and discuss the  Bohr inequality for planar harmonic functions as a special case. We expect that our approach will lead to several new investigations on the
general notion of the so-called Bohr's phenomenon.

\section{Bohr inequality for analytic functions of the form $\sum_k b_k z^{kp+m}$}\label{KP2-sec2a}

\begin{thm}\label{KP2-th3}
Let $p \in \mathbb{N}$ and $0\leq m \leq p$, $f(z)=\sum_{k=0}^{\infty} a_{pk+m}z^{pk+m}$ be analytic in $\ID$
and $|f(z)| \leq 1$ in $\ID$. Then
$$M_f(r) \leq 1 \mbox{ for } r \leq r_{p,m},
$$
where $ r_{p,m}$ is the maximal positive root of the equation
\be\label{KP3-eq4}
-6 r^{p-m} + r^{2(p-m)} + 8 r^{2p} +1= 0.
\ee
 The extremal function has the form $z^m(z^p-a)/(1-az^p)$, where
$$a=\left (1-\frac{\sqrt{1-{r_{p,m}}^{2p}}}{\sqrt{2}}\right )\frac{1}{{r_{p,m}}^p }.
$$
\end{thm}

The case $p=2$ and $m=1$ has a special interest which is indeed Theorem \Ref{KP2-odd1} and it provides a solution to the conjecture of
Ali, Barnard and Solynin \cite{AliBarSoly}. More generally, it is a simple exercise to see that for the case $p=m,2m,3m$, the Bohr radii give
$$r_{m,m}=1/\sqrt[2m]{2}, ~r_{2m,m}= \sqrt[m]{r_2}, ~\mbox{ and }~r_{3m,m}=\sqrt[2m]{\frac{7 + \sqrt{17}}{16}},
$$
respectively, where $r_2$ is given in Theorem \Ref{KP2-odd1}. It is worth pointing out from the last case that $r_{3,1}$ gives the value $(\sqrt{7 + \sqrt{17}})/4$. The result for $m=0$ is well known \cite{AliBarSoly} which we now recall because of its independent interest.

\begin{cor} \label{Lemma-1} %
Let $p\geq 1$. If $f(z)=\sum_{k=0}^\infty a_{pk} z^{pk}$ is analytic in $\mathbb{D}$, and $\left|f(z)\right|\le 1$ in $\mathbb{D}$, then $M_f(r)\le 1$
for $0\leq r\leq r_{p,0}=1/\sqrt[p]{3}$.  The radius $r_{p,0}=1/\sqrt[p]{3}$ is best possible.
\end{cor}

If $a_0=0$, it follows from the proof of Theorem \ref{KP2-th3} that the number $r_{p,0}=1/\sqrt[p]{3}$ in Corollary \ref{Lemma-1} can be evidently replaced
by $r_{p,0}=1/\sqrt[2p]{2}$ which is the Bohr radius in this case.  Moreover, the radius $r=1/\sqrt[2p]{2}$ in this case is best possible as
demonstrated by the function
$$\varphi_\alpha (z)=z^p\left (\frac{\alpha-z^p}{1-\alpha z^p}\right )
$$
with $\alpha =1/\sqrt[2p]{2}.$ The case $p=1$ of this result coincides with $p=2$ of a general result in
Corollary \ref{KP3-cor3} (especially when $g(z)\equiv 0$).

%

%
%

%


\section{Proof of Theorem \ref{KP2-th3}}\label{KP2-sec2}

For the proof of Theorem \ref{KP2-th3}, we need the following lemmas.

\begin{lem}\label{KP2-lem1}
Let $0\leq m \leq p$ and $r_{p,m}$ be as in Theorem \ref{KP2-th3}. Then $2{r_{p,m}}^{p+m} \leq 1$.
\end{lem}
\begin{proof} Let $y=r_{p,m}^{p+m}$. Then \eqref{KP3-eq4} becomes a quadratic equation
$$(8 +1/r_{p,m}^{2m})y^2-6y+r_{p,m}^{2m}=0
$$
which has two solutions
$$
y=\frac{3 \pm 2\sqrt{2}\sqrt{1-r_{p,m}^2}}{8 +1/r_{p,m}^2} \leq \frac{3 +2\sqrt{2}\sqrt{1-r_{p,m}^2}}{8 +1/r_{p,m}^2}.
$$
Consequently,
$$2r_{p,m}^{p+1}=2y \leq \frac{6 +4\sqrt{2}\sqrt{1-r_{p,m}^2}}{8 +1/r_{p,m}^2} \leq \sup_{r \in (0,1]}\frac{6 +4\sqrt{2}\sqrt{1-r^2}}{8 +1/r^2}=1,
$$
which completes the proof of Lemma \ref{KP2-lem1}.
\end{proof}

\begin{lem}\label{New} Let $0\leq m \leq p$ and $r_{p,m}$ be as in Theorem \ref{KP2-th3}. Then
$$\frac{1}{r^{p-m}}\left(3-2 \sqrt{2}\sqrt{1-r^{2p}} \right)=1.
$$
\end{lem}
\begin{proof} At first, suppose that $m<p$ and let $y=r^{p-m}$. Then \eqref{KP3-eq4} reduces to a quadratic equation
$$-6 y + y^2 + 8 r^{2p} +1= 0,
$$ which has two solutions $y_1=3 +2 \sqrt{2}\sqrt{1-r^{2p}}$ and $y_2=3 -2 \sqrt{2}\sqrt{1-r^{2p}}$.
The solution $y=y_1$ is impossible because all positive roots of the initial equation are less than $1$. Consequently
$y=y_2$.

Now, consider the case $m=p$. In this case $r_{m,m}=1/\sqrt[2m]{2}$ so that
$$ \frac{1}{r^{p-m}}\left(3-2 \sqrt{2}\sqrt{1-r^{2p}} \right)=3-2 \sqrt{2}\sqrt{1-r^{2m}}=3-2 \sqrt{2}\sqrt{1-1/2}=1
$$
and the proof is complete.
\end{proof}

We now recall the following lemma from \cite{KayPon1}.

\begin{Lem}\label{KP2-lem2}
 Let $|a|<1$ and $0 < R \leq 1$. If $g(z)=\sum_{k=0}^{\infty} b_kz^k$ is analytic and satisfies the inequality $|g(z)| \leq 1$ in $\ID$, then
the following sharp inequality holds:
\begin{equation}\label{KP2-eq3}
\sum_{k=1}^\infty |b_k|^2R^{pk} \leq R^{p}\frac{(1-|b_0|^2)^2}{1-|b_0|^2R^{p}}.
\end{equation}
\end{Lem}

\begin{proof}[Proof of Theorem {\rm \ref{KP2-th3}}]
The case $m=p$ follows from known result (\cite{KayPon1}) and thus, we may assume that $0\leq m < p$.
Following the idea from \cite{KayPon1}, we consider $f(z)=\sum_{k=0}^{\infty} a_{pk+m} z^{pk+m}$, where $|f(z)|\le 1$ for $z\in\ID$.
Also, let $r=r_{p,m}$. At first, we remark that the function $f$ can be represented as $f(z)=z^m g(z^p)$, where $|g(z)| \leq 1$ in $\ID$ and
$g(z)=\sum_{k=0}^{\infty} b_kz^k$ is analytic in $\ID$  with $b_k=a_{pk+m}$.  Let $|b_0|=a$. Choose any $\rho >1$ such that $\rho r \leq 1$. Then it follows that
\begin{eqnarray*}
\sum_{k=1}^{\infty} |a_{pk+m}|r^{pk} &=& \sum_{k=1}^{\infty}|b_k| r^{pk}  \\
&\leq & \sqrt{\sum_{k=1}^{\infty}|b_k|^2 \rho^{pk}r^{pk}} \sqrt{\sum_{k=1}^{\infty} \rho^{-pk}r^{pk}} \\
&\leq & \sqrt{r^p\rho^p\frac{(1-a^2)^2}{1-a^2r^p\rho^p}}\, \sqrt{\frac{\rho^{-p}r^p}{1-\rho^{-p}r^{p}}}\\
&= &  \frac{r^p(1-a^2)}{\sqrt{1-a^2r^p\rho^p}}\,  \frac{1}{\sqrt{1-\rho^{-p}r^{p}}}.
\end{eqnarray*}

In the second and the third steps above we have used the classical Cauchy-Schwarz inequality and \eqref{KP2-eq3} with $R=\rho r$, respectively.
 Hence
\begin{equation} \label{Additional}
\sum_{k=1}^{\infty} |a_{pk+m}|r^{pk} \leq \frac{r^p(1-a^2)}{\sqrt{1-a^2r^p\rho^p}}\,  \frac{1}{\sqrt{1-\rho^{-p}r^{p}}}.
\end{equation}

We need to consider the cases $a \ge r^p$ and $a<r^p$ separately.

\noindent
\vspace{6pt}
{\bf Case 1: $a \ge r^p$.} In this case set $\rho=1/\sqrt[p]{a}$ and
obtain
\begin{equation}\label{KP2-eq4}
\sum_{k=0}^{\infty}|a_{pk+m}|r^{pk+m} \leq r^m \left(a+r^p\frac{(1-a^2)}{1-r^p a}\right).
\end{equation}
For convenience, we may let $\alpha =r^p$ and consider
$$\psi (x)=x+\alpha  \frac{(1-x^2)}{1- \alpha x}, \quad x\in [0,1].
$$
As in \cite{KayPon1}, we find that $\psi (x) $ attains its maximum at $x=x_1$, where
$$x_1=\left (1-\frac{\sqrt{1-\alpha ^2}}{\sqrt{2}}\right )\frac{1}{\alpha }, \quad \alpha \ge \frac{1}{3},
$$
and thus, $\psi (x)\leq \psi (x_1)$. Consequently, by \eqref{KP2-eq4},  we find for the $r=r_{p,m}$ defined in Theorem \ref{KP2-th3} that
\begin{equation} \label{Finish1}
\sum_{k=0}^{\infty}|a_{pk+m}|r^{pk+m} \leq \frac{1}{r^{p-m}}\left(3-2 \sqrt{2}\sqrt{1-r^{2p}} \right)=1,
\end{equation}
where we have used Lemma \ref{New}.

\noindent
\vspace{6pt}
{\bf Case 2: $a < r^p$.} In this case we set $\rho=1/r$ and apply the inequality \eqref{Additional}. As a result we get
\begin{equation} \label{Finish2}
\sum_{k=0}^{\infty} |a_{pk+m}|r^{pk+m} \leq r^m(a+r^p \sqrt{1-a^2}/\sqrt{1-r^{2p}}) \leq 2r^{p+m} \leq 1.
\end{equation}
Here we omitted the critical point $a=\sqrt{1-r^{2p}}$ because it is less than or equal to $r^p$ only in the case $r^{2p}>1/2$ which
contradicts Lemma \ref{KP2-lem1}.

The last inequality in \eqref{Finish2} is derived in Lemma \ref{KP2-lem1}.

Inequalities (\ref{Finish1}) and (\ref{Finish2}) finish the proof of the first part of Theorem \ref{KP2-th3}. Now we have to say a few words about extremal.
We set $f(z)=z^m(z^p-a)/(1-az^p)$ with
$a=r^{-p}\left (1-\frac{\sqrt{1-{r}^{2p}}}{\sqrt{2}}\right )$ and then calculate the Bohr radius for it. It coincides with $r$.

Certainly, an extremal function is unique up to a rotation of $a$. To see this we just trace our inequalities and
see that the equality holds only when $|b_0|=a$.
\end{proof}

\section{The $p$-Bohr radius of bounded harmonic functions}\label{KP3-sec4}

Suppose that $f=h+\overline{g}$ is a harmonic mapping of $\mathbb{D}$ such that  $|f(z)|<1$ with
$h(z)=\sum_{k=0}^{\infty}a_{k}z^{k}$ and $g(z)=\sum_{k=1}^{\infty}b_{k}z^{k}$. Then $|a_{0}|<1$. Also, following
the proof from \cite{CPW2-jmaa} (see also \cite{CPW2-fil}), for $\theta\in [0, 2\pi )$, let
$$v_{\theta}(z)={\rm Im\,}(e^{i\theta}f(z))
$$
and observe that
$$v_{\theta}(z)  = {\rm Im\,}( e^{i\theta}h(z)+\overline{e^{-i\theta}g(z)})
={\rm Im\,}(e^{i\theta}h(z)-e^{-i\theta}g(z)).
$$
Because $|v_{\theta}(z)| <1$, it follows that
\be\label{KP3-eq1}
e^{i\theta}h(z)-e^{-i\theta}g(z)\prec K(z),
\ee
where $\prec$ denotes the usual subordination (cf. \cite{Abu,AAPon1,KayPon1}),
$$K(z)=\lambda+\frac{2}{\pi}\log \left (\frac{1+\xi z}{1-z} \right )=\lambda +\frac{2}{\pi}\sum_{k=1}^\infty \frac{(1-(-\xi)^k)}{k} z^n,
$$
$\xi=e^{-i\pi\mbox{Im}(\lambda)}$ and $\lambda=e^{i\theta}h(0)=e^{i\theta}a_0$. Note that
$K(z)$ maps $\ID$ onto a convex domain with $K(0)=\lambda$
and $K'(0)=\frac{2}{\pi}(1+\xi)$, and therefore, by the definition of subordination, it follows that
$$e^{i\theta}h(z)-e^{-i\theta}g(z) = K(w(z)) =\lambda+\frac{2}{\pi}\log \left (\frac{1+ \xi w(z) }{1-w(z)} \right ),
$$
where $w$ is analytic in $\ID$  satisfying $w(0)=0$ and $|w(z)|<1$ for $z\in\ID$.

Now, we may set $a_0=0$. Then $\lambda =0$ and hence, $\xi=1$. In this case, \eqref{KP3-eq1} reduces to
\be\label{KP3-eq2}
e^{i\theta}h(z)-e^{-i\theta}g(z)\prec \frac{2}{\pi}\log \left (\frac{1+ z}{1-z} \right )=\frac{4}{\pi}\sum_{k=1}^\infty \frac{1}{2k-1} z^{2k-1}.
\ee
Thus,
$$|h(z)-e^{-2i\theta}g(z) | \leq \frac{2}{\pi}\log\left (\frac{1+|z|}{1-|z|} \right )
$$
and hence,
$$|h(z)|+|g(z)| \leq  \frac{2}{\pi}\log \left ( \frac{1+|z|}{1-|z|}\right )
$$
as a consequence of the arbitrariness of $\theta$ in $[0, 2\pi )$. Thus, we have

\begin{thm}\label{KP3-th1}
Let $f=h+\overline{g}$ be a harmonic mapping in $\ID$ such that $f(0)=h(0)=0$ and $|f(z)| < 1$ in $\ID$.
Then $|h(z)|+|g(z)| \leq 1$
for $|z| \leq r_0=\tanh(\pi/4)$. The number $r_0$ is sharp as the functions
$$f_0(z)=\frac{2e^{i\alpha}}{\pi}\mbox{\rm Im}\left(\log\frac{1+e^{i\beta} z}{1-e^{i\beta} z}\right)
$$
show, whose values are confined to a diametral segment of the disk $\ID$. Here $\alpha$ and $\beta$ are real numbers.
\end{thm}

In view of Theorem \ref{KP3-th1} it is natural to ask question about the Bohr radius for a pair $(h,g)$ of  analytic functions in the disk $\ID$
that satisfies the inequality $|h(z)|+|g(z)| \leq 1$ in $\ID$. This provides several new information on the Bohr radius.

\bprop\label{prop1}
Assume that $h(z)=\sum_{k=0}^\infty a_k z^k$ and $g(z)=\sum_{k=0}^\infty b_k z^k$ are  two analytic functions in the disk $\ID$
satisfying the condition $|h(z)|+|g(z)| \leq 1$ in $\ID$. If $g(0)=0$, then we have
\be\label{KP3-eq8}
|a_0|+\sum_{k=1}^\infty (|a_k|+|b_k|)r^k \leq 1 ~\mbox{ for $|z|=r \leq 1/3$}.
\ee
\eprop
\bpf
The proof is almost identical to the classical proof of Bohr's theorem. Let us consider the function $\Phi(z)=h(z)+e^{i\theta}g(z)$,
where $\theta$ is arbitrary. This function satisfies the inequality $|\Phi(z)| \leq 1$ in the unit disk $\ID$ and
consequently, $|a_n+e^{i\theta}b_n| \leq 1-|a_0|^2$. Hence, $|a_n|+|b_n| \leq 1-|a_0|^2$ for $n\geq 1$.
It means that
$$|a_0|+\sum_{k=1}^\infty (|a_k|+|b_k|)r^k \leq |a_0|+(1-|a_0|^2)\frac{r}{1-r} \leq 1 ~\mbox{ for }~ r \leq 1/3,
$$
and thus, \eqref{KP3-eq8} follows.

A natural question is to whether the inequality \eqref{KP3-eq8} is true if we replace $|a_0|$ in \eqref{KP3-eq8} by $|a_0|+|b_0|$.
It turns out that this proposition is false without an additional information on $g$ or $h$. To prove this, we let $a>0$ and consider the functions
$$h(z)=\frac{z+a}{2\sqrt{1+a^2}}, \quad g(z)=\frac{z-a}{2\sqrt{1+a^2}}
$$
so that $h(0)=a/(2\sqrt{1+a^2})$ and $h(0)=-g(0)$. Also, we have
\beqq
(|h(z)|+|g(z)|)^2&=&|h(z)|^2+|g(z)|^2+2|h(z)||g(z)|\\
&=&\frac{2(|z|^2+a^2)+2|z^2-a^2|}{4(1+a^2)} \\
&\leq & \frac{2(|z|^2+a^2)+2(|z|^2+a^2)}{4(1+a^2)} < 1
\eeqq
for $z \in \ID$. On the other hand, for this function, we obtain that
$$\sum_{k=0}^\infty (|a_k|+|b_k|)r^k = \frac{r+a}{\sqrt{1+a^2}} > 1 \mbox{ for } r>r_0(a)=\sqrt{1+a^2}-a.
$$
But,
$$r_0(a)=\frac{1}{\sqrt{1+a^2}+a} \to 0 \mbox{ as } a \to +\infty.
$$
This example shows that Proposition \ref{prop1} has a precise form and that the condition $g(0)=0$ cannot be ignored.
Otherwise, there will be no positive Bohr radius. Thus, $|a_0|$ in \eqref{KP3-eq8} cannot be replaced by $|a_0|+|b_0|$.
\epf

However, there is a generalization of the classical Bohr theorem in which the initial condition $g(0)=0$ in Proposition \ref{prop1}
can be omitted.

\bprop\label{prop2}
Assume that $h(z)=\sum_{k=0}^\infty a_k z^k$ and $g(z)=\sum_{k=0}^\infty b_k z^k$ are  two analytic functions in the disk $\ID$
satisfying the condition $|h(z)|+|g(z)| \leq 1$ in $\ID$. Then, we have
\be\label{KP3-eq6}
\sqrt{|a_0|^2+|b_0|^2}+ \sum_{k=1}^\infty \left (\sqrt{|a_k|^2+|b_k|^2}\right )r^k \leq 1 \mbox{ for } r \leq 1/3.
\ee
\eprop\bpf
We again consider the function $\Phi(z)=h(z)+e^{i\theta}g(z)$ so that $|\Phi(z)| \leq 1$ for $z\in\ID$ and consequently,
the classical Wiener inequality gives
$$|a_n+e^{i\theta}b_n| \leq 1-|a_0+e^{i\theta}b_0|^2  ~\mbox{ for $n\ge 1$}.
$$
This inequality together with the inequality $|a_n-e^{i\theta}b_n| \leq 1-|a_0-e^{i\theta}b_0|^2$ gives
$$|a_n+e^{i\theta}b_n| +|a_n-e^{i\theta}b_n| \leq 2-|a_0+e^{i\theta}b_0|^2 -|a_0-e^{i\theta}b_0|^2=2(1-|a_0|^2-|b_0|^2).
$$
From here, we derive the following inequality
$$(|a_n+e^{i\theta}b_n| +|a_n-e^{i\theta}b_n|)^2 \leq 4(1-|a_0|^2-|b_0|^2)^2
$$
which is equivalent to the inequality
$$|a_n|^2+|b_n|^2+|a_n^2-e^{i 2\theta}b_n^2| \leq 2(1-|a_0|^2-|b_0|^2)^2.
$$
Now, we choose $\theta=\pi/2+\arg a_n-\arg b_n$ and obtain
$$|a_n|^2+|b_n|^2 \leq (1-|a_0|^2-|b_0|^2)^2
$$
which gives  $\sqrt{|a_n|^2+|b_n|^2} \leq 1-|a_0|^2-|b_0|^2$ for $n\geq 1$. Therefore,
\beqq
\sum_{k=0}^\infty \left (\sqrt{|a_k|^2+|b_k|^2}\right )r^k &\leq &  \sqrt{|a_0|^2+|b_0|^2} + (1-|a_0|^2-|b_0|^2)\frac{r}{1-r}\\
& \leq & \sqrt{|a_0|^2+|b_0|^2} + (1-\sqrt{|a_0|^2+|b_0|^2})\frac{2r}{1-r} \\
&\leq 1& ~\mbox{ for $r \leq 1/3$}
\eeqq
and the proof is completed.
\epf

We observe that if the constant term $\sqrt{|a_0|^2+|b_0|^2}$ in \eqref{KP3-eq6} is replaced by $|a_0|^2+|b_0|^2$, then the number $1/3$ in
\eqref{KP3-eq6} can be replaced by $1/2$.

\br
{\rm
If $f=h+\overline{g}$ is a harmonic mapping of $\mathbb{D}$ such that  $|f(z)|<1$  and $h(0)=0=f(0)$, then the
subordination relation \eqref{KP3-eq2} gives the inequality
\be\label{KP3-eq3}
\sum_{k=1}^\infty |a_k -e^{-2i\theta}b_k|^2r^{2k} \leq \frac{16}{\pi ^2}\sum_{k=1}^\infty \frac{1}{(2k-1)^2}r^{2(2k-1)}
\ee
and the arbitrariness of $\theta$ in $[0, 2\pi )$ shows that
$$\sum_{k=1}^\infty (|a_k|^2 +|b_k|^2)r^{2k} \leq \frac{16}{\pi ^2}\sum_{k=1}^\infty \frac{1}{(2k-1)^2}r^{2(2k-1)}.
$$
From \eqref{KP3-eq3}, we also have
$$\sum_{k=1}^\infty |a_k -e^{-2i\theta}b_k|r^{k} \leq
\frac{4}{\pi }\sqrt{\sum_{k=1}^\infty \frac{1}{(2k-1)^2}(r\rho)^{2(2k-1)}} \left ( \frac{1}{\sqrt{\rho^{2}-1}}\right ),
$$
where $\rho >1$ such that $\rho r \leq 1$. The choice $\rho=1/r$ gives
\beqq
\sum_{k=1}^\infty |a_k -e^{-2i\theta}b_k|r^{k}  &\leq &
\frac{4}{\pi }\sqrt{\sum_{k=1}^\infty \frac{1}{(2k-1)^2}}  \frac{r}{\sqrt{1-r^2}}=\frac{4}{\pi }\sqrt{\frac{\pi^2}{8}}  \frac{r}{\sqrt{1-r^2}}=\frac{\sqrt{2}r}{\sqrt{1-r^2}}.
\eeqq
}\er

However, one can also obtain general results. To do this, let us recall some facts.
In the paper \cite{Abu} the Bohr radius for harmonic functions $f=h+\overline{g}$ was defined as the maximal $r$ for which
$$\sum_{k=1}^\infty {(|a_k| +|b_k|)}r^{k} \leq 1.
$$
However, one can give a more flexible definition of the Bohr radius in the case of harmonic functions.
For  $p \ge 1$, we say that $r_p$ is the $p$-Bohr radius for the harmonic function $f=h+\overline{g}$ of the unit disk $\ID$, where $g(0)=0$,
if $r_p$ is the largest value such that
$$\sum_{k=0}^\infty(|a_k|^p +|b_k|^p)^{1/p}r^{k} \leq 1  ~\mbox{ for $|z|=r \leq r_p$}.
$$
At first we remark that all these radii coincide in the analytic case. Moreover, essentially, they are equivalent in view of the estimates
$$\max(|a_k|,|b_k|) \leq (|a_k|^p +|b_k|^p)^{1/p}  \leq |a_k|+|b_k| \leq 2 \max(|a_k|,|b_k|).
$$
Note that the case $\max(|a_k|,|b_k|)$ corresponds to $p=\infty$. This approach can be used to prove the following general result.

\begin{thm}\label{KP3-th3}
Let $f=h+\overline{g}$ be a harmonic mapping in $\ID$ such that $|f(z)| \leq 1$ in $\ID$, where
$h(z)=\sum_{k=0}^{\infty}a_{k}z^{k}$ and $g(z)=\sum_{k=1}^{\infty}b_{k}z^{k}$.
Then for any $p \ge 1$ and $r<1$, the following inequality holds:
\beqq
\sum_{k=1}^\infty(|a_k|^p +|b_k|^p)^{1/p}r^{k}  &\leq &
 \max \{2^{(1/p)-1/2},1\}\sqrt{1-|a_0|^2}\frac{r}{\sqrt{1-r^2}}.
\eeqq
\end{thm}
\bpf
We have $ 1 \ge |h+\overline{g}|^2=(h+\overline{g})(\overline{h}+g)=|h|^2+|g|^2+2 {\rm Re}\, (hg)$
and thus,
\beqq
1 &\geq & \frac{1}{2\pi}\int_0^{2\pi}  |f(re^{i\theta})|^2\, d\theta
=  \frac{1}{2\pi} \int_0^{2\pi} ( |h(re^{i\theta})|^2+|g(re^{i\theta})|^2)\, d\theta\\
 &=& |a_0|^2+\sum_{k=1}^\infty (|a_k|^2 +|b_k|^2)r^{2k}
\eeqq
so that letting $r$ approach $1$, we get
$$|a_0|^2+\sum_{k=1}^\infty (|a_k|^2 +|b_k|^2) \leq 1.
$$
As a result of it, we obtain
\beqq
\sum_{k=1}^\infty (|a_k|^p +|b_k|^p)^{1/p}r^{k} &\leq &\sqrt{\sum_{k=1}^\infty(|a_k|^p +|b_k|^p)^{2/p}}{\sqrt{\sum_{k=1}^\infty r^{2k}}}\\
&\leq& \sqrt{\max\{ 2^{(2/p)-1}, 1\}\sum_{k=1}^\infty(|a_k|^2 +|b_k|^2)}\, \frac{r}{\sqrt{1-r^2}}\\
& \leq & \max\{2^{(1/p)-1/2},1\}\sqrt{1-|a_0|^2}\frac{r}{\sqrt{1-r^2}}
\eeqq
and the proof is complete.
\epf

In particular, Theorem \ref{KP3-th3} for $p=1$ shows that for $r\leq 1/3$,
$$\sum_{k=1}^\infty (|a_k| +|b_k|) r^k \leq \frac{\sqrt{1-|a_0|^2}}{2}\leq \frac{1}{2}.
$$
It is worth pointing out that Abu Muhanna \cite[Theorem 1]{Abu} obtained a similar inequality with $2/\pi \approx 0.63662$
instead of $(\sqrt{1-|a_0|^2})/2$. This observation shows that the above inequality is a vast improvement over the result of \cite{Abu}.

Note that $\max \{2^{(1/p)-1/2},1\}$ equals $1$ for $p\geq 2$, and equals $2^{(1/p)-1/2}$ for $p\in [1,2]$. This observation for
$a_0=0$ in Theorem \ref{KP3-th3} gives the following.

\begin{cor}\label{KP3-cor3}
Suppose that $f=h+\overline{g}$ is a harmonic mapping of $\ID$ such that $|f(z)| \leq 1$ in $\ID$, where
$h(z)=\sum_{k=1}^{\infty}a_{k}z^{k}$ and $g(z)=\sum_{k=1}^{\infty}b_{k}z^{k}$. If $p \ge 2$. Then
$$
\sum_{k=1}^\infty (|a_k|^p +|b_k|^p)^{1/p}r^{k} \leq 1 \mbox{ for } r \leq \frac{1}{\sqrt{2}}.
$$
The number $1/\sqrt{2}$ is sharp.
\end{cor}


\begin{cor}\label{KP3-cor4}
Suppose that $f=h+\overline{g}$ is  a harmonic mapping in $\ID$ such that $|f(z)| \leq 1$ in $\ID$, where
$h(z)=\sum_{k=0}^{\infty}a_{k}z^{k}$ and $g(z)=\sum_{k=1}^{\infty}b_{k}z^{k}$. If $p \in [1,2]$ and $|a_0|<1$, then
\be\label{KP3-eq7}
|a_0|+\sum_{k=1}^\infty \left (|a_k|^p +|b_k|^p\right )^{1/p}r^{k} \leq 1 \mbox{ for } r \leq r_p(|a_0|)
=\sqrt{\frac{1-|a_0|}{2^{(2/p)-1}+1+(2^{(2/p)-1}-1)|a_0|}}.
\ee
In the case $p \ge 2$ the following inequality holds:
\begin{equation} \label{Sharp}
|a_0|+\sum_{k=1}^\infty \left (|a_k|^p +|b_k|^p\right )^{1/p}r^{k} \leq 1 \mbox{ for } r \leq \sqrt{\frac{1-|a_0|}{2}}.
\end{equation}
\end{cor}
\bpf
Clearly, for $p \in [1,2]$ and $|a_0|<1$,  Theorem \ref{KP3-th3} gives that
$$|a_0|+\sum_{k=1}^\infty \left (|a_k|^p +|b_k|^p\right )^{1/p}r^{k} \leq |a_0|+
2^{(1/p)-1/2}\,\sqrt{1-|a_0|^2}\frac{r}{\sqrt{1-r^2}}
$$
which is less than or equal to $1$ if  $ r \leq r_p(|a_0|)$. This gives the condition \eqref{KP3-eq7}. The case  $p \ge 2$
follows similarly.
\epf

For the harmonic functions, the case $p=1$ of \eqref{KP3-eq7} gives
$$|a_0|+\sum_{k=1}^\infty (|a_k| +|b_k|)r^{k} \leq 1 \mbox{ for } r \leq r_1(|a_0|)
=\sqrt{\frac{1-|a_0|}{3 +|a_0|}}.
$$
Also, we remark that in the case $|a_0|=1$ in Corollary \ref{KP3-cor4} we have $a_k=b_k=0$ for all $k \ge 1$.

\br
{\rm
The example
$$f(z)=\frac{z+a_0}{1+\overline{a_0} z}=\sum_{k=0}^\infty a_k z^k
$$
shows that in the case $p \ge 2$ and $|a_0|=1/2$ the inequality (\ref{Sharp}) is sharp for the analytic case. This fact follows from the identity
$$\sum_{k=0}^\infty |a_k| r^k=|a_0|+\frac{r(1-|a_0|^2)}{1-r|a_0|}
$$
which implies that the Bohr radius for this function is $1/(1+2|a_0|).$
}
\er

We know that in the classical case the Bohr radius is $1/3$. In harmonic case, the $p$-Bohr radius can be also $1/3$
under an additional assumption on the constant coefficient $a_0$.

\begin{cor}\label{KP3-cor5}
Let $f=h+\overline{g}$ be a harmonic mapping in $\ID$ such that $|f(z)| \leq 1$ in $\ID$, where
$h(z)=\sum_{k=0}^{\infty}a_{k}z^{k}$ and $g(z)=\sum_{k=1}^{\infty}b_{k}z^{k}$. If $p \in [1, 2]$ and
\be\label{KP2-eq5}
|a_0| \leq A(p)=\frac{8-2^{(2/p)-1}}{8+2^{(2/p)-1}},
\ee
then
$$|a_0|+\sum_{k=1}^\infty (|a_k|^p +|b_k|^p)^{1/p}r^{k} \leq 1 \mbox{ for } r \leq 1/3.
$$
\end{cor}
\bpf Suffices to observe that $r_p(|a_0|)$ given in Corollary \ref{KP3-cor4} shows that
$r_p(|a_0|)\geq 1/3$ whenever $|a_0|$ satisfies the inequality  \eqref{KP2-eq5}. The desired conclusion follows.
\epf

\br
{\rm
Since $|a_0|\leq 1$, the question is whether the range for $a_0$ in Corollary \ref{KP3-cor5} could be enlarged.
Note also that $A(p)$ is an increasing function of $p$ and thus, $A(p)\in [3/5, 7/9]$ for $p \in [1, 2]$. We now show that
in the case of $p=1$, the value $A(1)$ cannot be greater than $0.67404$ so that the number $3/5=0.6$ is quite close to be sharp.
To see this, we will use an example from the paper \cite{Abu}:
$$ f(z)= \frac{2}{\pi}\mbox{\rm Im}\left(\log\frac{1+ z}{1-z}\right)\sin \mu + i\cos\mu
$$
which may be rewritten as
$$ f(z)= a_0+\sum_{k=1}^{\infty}a_{k}z^{k}  +\sum_{k=1}^{\infty}b_{k}\overline{z^{k}}
$$
where
$$a_0= i\cos\mu ~\mbox{ and }~ a_k=\frac{2\sin \mu}{\pi (2k-1)}=-b_k \mbox{ for $k\geq 1$}.
$$
Consequently,
$$|a_0|+\sum_{k=1}^\infty (|a_k|^p +|b_k|^p)^{1/p}r^{k}= |\cos\mu|+ \frac{2^{1/p}}{\pi}\log\left(\frac{1+ r}{1-r}\right)|\sin \mu| .
 $$
Thus, for the case $r=1/3$, the last expression is greater than or equal to $1$ whenever
$$ |a_0|+ \frac{2^{1/p}\,\log 2}{\pi} \sqrt{1-|a_0|^2}   \geq 1
$$
which gives the condition
$$|a_0|=|\cos\mu| \ge \frac{\pi^2-2^{2/p} \log^2 2}{\pi^2+2^{2/p} \log^2 2}.
$$
This observation means that
$$A(p) \leq \frac{\pi^2-2^{2/p} \log^2 2}{\pi^2+2^{2/p} \log^2 2}.
$$
For instance, $0.6 \leq A(1) \leq 0.67404$ and $\frac{7}{9} \leq A(2) \leq 0.82256$.
}
\er

Finally, as a harmonic analog of Theorem \Ref{KP2-odd1}, we have the following result in a general form and the proof of
it follows from the lines of proof of Theorem \ref{KP3-th3}. So we omit its detail.

\begin{thm}\label{KP3-th4}
Let $f(z)=\sum_{k=0}^{\infty}a_{2k+1}z^{2k+1} + \sum_{k=0}^{\infty}\overline{b_{2k+1}z^{2k+1}}$ be
an odd harmonic function in $\ID$ such that $|f(z)| <1$ in $\ID$.
Then for $p \in [1,2]$ we have the following estimate:
$$\sum_{k=0}^\infty (|a_{2k+1}|^p +|b_{2k+1}|^p)^{1/p}r^{2k+1}\leq  2^{(1/p)-1/2}\frac{r}{\sqrt{1-r^4}} \leq 1
~\mbox{ for }~ r \leq \rho_p
$$
where $\rho_p =\sqrt{\sqrt{4^{(2/p)-2}+1}-2^{(2/p)-2}}. $ In the case $p \ge 2$ the following inequality holds:
$$\sum_{k=0}^\infty (|a_{2k+1}|^p +|b_{2k+1}|^p)^{1/p}r^{2k+1}\leq  \frac{r}{\sqrt{1-r^4}} \leq 1 ~\mbox{ for }~
r \leq \rho_2=\sqrt{\frac{\sqrt{5}-1}{2}}.
$$
\end{thm}

Let us remark that the number $\rho_1=\sqrt{\sqrt{2}-1}=0.6435 \ldots $ cannot be replaced by a number greater than
$$r_0=\frac{e^{\pi/2}-1}{e^{\pi/2}+1}=\tanh \frac{\pi}{4}=0.65579\ldots
$$
Thus, it is natural to conjecture that the Bohr radius in the case of harmonic odd function is
$r_0\approx 0.65579\ldots .$

Also, let us remark that $\rho_2=\sqrt{(\sqrt{5}-1)/2}=0.7861\ldots $. This number cannot be more than  $0.789991$.

\subsection*{Acknowledgements}
The research of the first author was supported by Russian foundation for basic research, Proj. 17-01-00282, and the
research of the second author was supported by the project RUS/RFBR/P-163 under Department of Science \& Technology (India).
The second author is currently on leave from  IIT Madras.



\begin{thebibliography}{99}


\bibitem{Abu} Y. Abu-Muhanna,
Bohr's phenomenon in subordination and bounded harmonic classes,
\emph{Complex Var. Elliptic Equ.} \textbf{55}(11) (2010),  1071--1078.

\bibitem{Abu2} Y. Abu-Muhanna and R. M. Ali,
Bohr's phenomenon for analytic functions into the exterior of a compact convex body,
\emph{J. Math. Anal. Appl.} \textbf{379}(2) (2011), 512--517.

\bibitem{Abu4} Y. Abu-Muhanna and R. M. Ali,
Bohr's phenomenon for analytic functions and the hyperbolic metric,
\emph{Math. Nachr.} \textbf{286}(11-12) (2013), 1059--1065.


\bibitem{Aizen-00-1} L. Aizenberg,
Multidimensional analogues of Bohr's theorem on power series,
\emph{Proc. Amer. Math. Soc.} \textbf{128}(4) (2000), 1147--1155.

\bibitem{Aizen-05-3} L. Aizenberg,
Generalization of Carath\'eodory's inequality and the Bohr radius for multidimensional power series,
in {\it Selected topics in complex analysis},
87--94, Oper. Theory Adv. Appl., \textbf{158}, Birkh\"auser, Basel, 2005.

\bibitem{AAD2} L. Aizenberg, A. Aytuna, and P. Djakov,
An abstract approach to Bohr's phenomenon,
Proc. Amer. Math. Soc. \textbf{128}(9) (2000), 2611--2619.

\bibitem{AizenTark-01-2} L. Aizenberg and N. Tarkhanov,
A Bohr phenomenon for elliptic equations,
\emph{Proc. London Math. Soc.}  \textbf{82}(2) (2001), 385--401.


\bibitem{AAPon1} R. M. Ali, Y. Abu-Muhanna, and S. Ponnusamy, On the Bohr inequality,
In ``Progress in Approximation Theory and Applicable Complex Analysis'' (Edited by N.K. Govil et al. ),
Springer Optimization and Its Applications \textbf{117} (2016), 265--295.

\bibitem{AliBarSoly}  R. M. Ali, R. W. Barnard, and A. Yu. Solynin,
A note on the Bohr's phenomenon for power series,
\emph{J. Math. Anal. Appl.} \textbf{449}(1) (2017), 154--167.


\bibitem{BaluCQ-2006} R. Balasubramanian, B. Calado, and H. Queff\'{e}lec,
The Bohr inequality for ordinary Dirichlet series,
\emph{Stud. Math.} \textbf{175}(3) (2006), 285--304.

\bibitem{BenDahKha} C. B\'{e}n\'{e}teau, A. Dahlner and D. Khavinson,
Remarks on the Bohr phenomenon,
\emph{Comput. Methods Funct. Theory} \textbf{4}(1) (2004), 1--19.

\bibitem{BoasKhavin-97-4} H. P. Boas and  D. Khavinson,
Bohr's power series theorem in several variables,
\emph{Proc. Amer. Math. Soc.} \textbf{125}(10) (1997),  2975--2979.

\bibitem{Bohr-14} H. Bohr,
A theorem concerning power series,
\emph{Proc. London Math. Soc.} \textbf{13}(2) (1914), 1--5.

\bibitem{CPW2-jmaa} SH. Chen, S. Ponnusamy and X. Wang,
Bloch constant and Landau's theorems for planar $p$-harmonic
mappings,
\emph{J. Math. Anal. Appl.}, {\bf 373}(2011), 102--110.

\bibitem{CPW2-fil} SH. Chen, S. Ponnusamy and X. Wang,
On some properties of solutions of the $p$-harmonic equation,
\emph{Filomat} \textbf{27}(4) (2013), 577--591.



\bibitem{KayPon1} I. R Kayumov and S. Ponnusamy, Bohr inequality for odd analytic functions,
\emph{Comput. Methods Funct. Theory}, (2017) Available online; 
DOI 10.1007/s40315-017-0206-2\\
See also {\tt https://arxiv.org/pdf/1701.03884.pdf}

\bibitem{PaulSingh-04-11} V. I. Paulsen and D. Singh,
Bohr's inequality for uniform algebras,
\emph{Proc. Amer. Math. Soc.} \textbf{132}(12) (2004),  3577--3579.

\bibitem{Pop} G. Popescu, Multivariable Bohr inequalities,
\emph{Trans. Amer. Math. Soc.} \textbf{359}(11) (2007), 5283--5317.

\end{thebibliography}
\end{document}